\documentclass{amsart}
\usepackage{amssymb,latexsym,enumitem,url}
\usepackage[noadjust]{cite}
\usepackage{graphicx}
\usepackage[latin1]{inputenc}
\usepackage[T1]{fontenc}
\usepackage{setspace}
\usepackage{color}
\usepackage{hyperref}
\usepackage{bbm}

\newcommand{\R}{\mathbb{R}}

\newcommand{\E}{\mathcal{E}}
\newcommand{\T}{\mathcal{T}}
\newcommand{\Sc}{\mathcal{S}}
\newcommand{\Exp}{\mathbb{E}}

\numberwithin{equation}{section}

\newtheorem{theorem}{Theorem}[section]
\newtheorem{corollary}[theorem]{Corollary}
\newtheorem{lemma}[theorem]{Lemma}
\newtheorem{proposition}[theorem]{Proposition}
\newtheorem{remark}[theorem]{Remark}
\newtheorem{definition}[theorem]{Definition}

\newtheorem{example}[theorem]{Example}

\allowdisplaybreaks

\title{Solutions of SPDE's  associated with
a stochastic flow}

\author{Suprio Bhar}
\address{Suprio Bhar, Tata Institute of Fundamental Research, Centre For Applicable Mathematics,Post Bag No 6503, GKVK Post Office, Sharada Nagar, Chikkabommsandra, Bangalore 560065, India.}
\email{suprio@tifrbng.res.in}

\author{Rajeev Bhaskaran}
\address{Rajeev Bhaskaran, 8th Mile Mysore Road, Indian Statistical Institute, Bangalore 560059, India.}
\email{brajeev@isibang.ac.in}

\author{Barun Sarkar}
\address{Barun Sarkar, 8th Mile Mysore Road, Indian Statistical Institute, Bangalore 560059, India.}
\email{barunsarkar.math@gmail.com}

\date{}

\begin{document}

\begin{abstract}
We consider the following stochastic partial differential equation,
\begin{align*}
&dY_t=L^\ast Y_tdt+A^\ast Y_t\cdot dB_t\\
&Y_0=\psi,
\end{align*}
associated with a stochastic flow $\{X(t,x)\}$, for $t \geq 0$, $x
\in \R^d$, as in [Rajeev \& Thangavelu, \emph{{Probabilistic
representations of solutions of the forward
  equations}}, Potential Anal. \textbf{28} (2008), no.~2, 139--162]. We show that the strong
solutions constructed there are `locally of compact support'. Using
this notion,we define the mild solutions of the above equation and
show the equivalence between strong and mild solutions in the multi
Hilbertian space $\Sc^\prime$. We show uniqueness of solutions in
the case when $\psi$ is smooth via the `monotonicity inequality' for
$(L^\ast,A^\ast)$, which is a known criterion for uniqueness.
\end{abstract}

\maketitle

\section{Introduction}\label{s:intro}
In this article we study the equation for the stochastic flow
generated by a finite dimensional diffusion $\{X(t,x)\}$ starting at
$x \in {\mathbb R}^d$ and satisfying a stochastic differential
equation with smooth coefficients viz.
\begin{equation}
\begin{split}
  &dX_t=\sigma(X_t)\cdot dB_t+b(X_t)dt\\
 &X_0=x.
\end{split}
\end{equation} We recall from \cite{MR2373102} that
the equation for the flow is given as
\begin{equation}\label{SPDE}
\begin{split}
dY_t&=L^\ast Y_tdt+A^\ast Y_t\cdot dB_t\\
Y_0&=\psi,
\end{split}
\end{equation}
where the operators $L^\ast$, $A^\ast$ are adjoints of the operators
$L$, $A$ respectively, associated with the diffusion and defined in
Section \ref{s:2} below. There the solutions of \eqref{SPDE} were
constructed in the space of distributions with compact support and a
fortiori, in some Hermite-Sobolev space $\Sc_p$ for some $p \in
{\mathbb R}$. Such solutions, given in terms of a representation of
$\psi$ as the derivative of continuous functions (see
\eqref{eq2},\eqref{eq4}) and $X(t,x)$ and its derivatives, can be
shown to take values in $\Sc_q$ for some $q < p$. In particular,
when $\psi$ is given by a function, then the solution is given by
\begin{equation}
Z_t(\psi) := \int_{\R^d}
\psi(x)\delta_{X(t,x)}~dx.
\end{equation}
While the set up of the
Hermite-Sobolev spaces facilitated the construction of solutions and
thus settled the question of existence, these same spaces turn out
to be difficult to handle when it comes to the question of
uniqueness of the solutions.

It was shown in \cite{MR2373102} that uniqueness of solutions for \eqref{SPDE} follows from the so called `Monotonicity inequality' for the
pair of operators $(L^\ast, A^\ast)$ (see \eqref{mono-ineq} below). However, the
multiplication operators that intervene in the definition of $L^\ast$
and $A^\ast$ create significant difficulties in proving these
inequalities because of their non self adjointness in these spaces.
Uniqueness for the Gaussian case can however be handled by special
methods (see \cite{Bhar2017}). For some background on the `Monotonicity
inequality' we refer to \cite{MR570795, MR1135324, MR2590157, MR2479730, MR3331916, MR3063763}. In this paper we
prove the Monotonicity inequality, in Section \ref{s:4},  in the self
adjoint case i.e. when the inequality holds in $\mathcal{L}^2({\mathbb R}^d)$
or $q = 0$ and when the initial condition belongs to $\Sc_p$, $p
\geq 5$. In particular the solutions are unique when the initial
condition $\psi$ is in $C^\infty_c({\mathbb R}^d) (\subset \Sc)$, the space of smooth functions with compact support.

We show that the solutions constructed in \cite{MR2373102} have an
important property viz. that they are `locally of compact support'
(see Definition \ref{loc-cpt-supp} and Proposition
\ref{specific-cpt-support}). This means that upto a stopping time
the supports of $Z_t(\psi)$ are contained in some compact set,
almost surely. In particular, the supports grow slowly enough that
such containment is possible. This contrasts sharply with the
expected value of such solutions which need not be of compact
support. Indeed this is precisely the behavior of the solutions of
the Cauchy problem for the Laplacian with initial value $\delta_0$,
the Dirac distribution at zero. The latter property is connected
with the stochastic representation of the solutions to the Cauchy
problem associated to $L^\ast$ (see \cite{MR2373102, MR1999259}).

    While the above discussion relates to `strong solutions' another
notion of solution for stochastic PDE's which is frequently used is
the notion of `mild solution'. There is an extensive literature on
solutions of SPDEs in the mild form in various function spaces (see
\cite{MR2560625, MR1207136, MR1489313,MR3222416}). However
there seems to be very little on mild solutions in the dual of a
countable Hilbertian space, the set up that we use (see however
\cite{MR3071506} equation (81), \cite[Chapter 3, Section
5]{MR771478},\cite{MR922982, MR1465436}, \cite[Theorem
6.1]{MR674449}, \cite[equation (3.11), p.314]{MR876085}). We show in
Section \ref{s:3} that the strong and mild solutions are equivalent
in this set up. The mild solutions of \eqref{SPDE} above, say
$\{Y_t\}$, are defined in terms of the dual $(S_t^\ast)$ of the
semi-group $(S_t)$ associated with the diffusion $\{X_t\}$. This
requires that the domain of $S_t^\ast$ be the distributions with
compact support. More over we can obtain good bounds on the operator
norms of $S_t^\ast$ only when the domain is restricted to
distributions with support in a fixed compact set (see \cite[Theorem
4.8]{MR2373102}). Thus a term like $S_t^\ast A^\ast Y_t$ in the
stochastic integral, in the definition of mild solution, is well
defined if the support of the process $\{Y_t\}$ is contained in a
fixed compact set. While this need not be the case in general, for
stochastic integrals to make sense, it is enough if this property
holds locally in time i.e. up to  a stopping time. Thus the notion
of a process which is `locally of compact support' appears quite
naturally in the definition of mild solutions in our set up.

We give two proofs that the strong solutions are also mild
solutions. One of them goes through an integration by parts formula
for the `product' $S^\ast_{t-s}Y_s$ ($t$ fixed, $s \leq t$) while the other uses an It\^o formula for the
function $G(s,y) := <x,S_{t-s}^\ast y>, x \in \Sc$ for fixed $x,t$ and
$s \leq t$.
The latter proof follows closely a well known computation in the
finite dimensional case to prove the martingale representation for
functionals of the form $f(X_t)$ by applying It\^o formula to $g(s,y)
:= S_{t-s}f(y)$. Indeed, it follows in fact that the mild solution
representation for finite dimensional diffusion is equivalent to the
martingale representation (see Proposition \ref{mld-mar}).

\section{Preliminaries}\label{s:2}
In this section, we describe the framework of our results which is
the same as that of \cite{MR2373102}, and recall the main results
from there. We also introduce the notion of distribution valued
processes which are `locally of compact support'.

Let $\Omega=C([0,\infty),\mathbb{R}^r)$ denote the set of continuous
functions on $[0,\infty)$ with values in $\mathbb{R}^r$. Let
$\mathcal{F}$ be the Borel $\sigma$-field on $\Omega$ and $P$ be the
Wiener measure. We denote $B_t(\omega):=\omega(t)$,
$\omega\in\Omega$, $t\geq0$ and recall that under $P$, $\{B_t\}$ is
a standard $r$ dimensional Brownian motion. Consider the following
stochastic differential equation
\begin{equation}\label{eq1}
\begin{split}
  &dX_t=\sigma(X_t)\cdot dB_t+b(X_t)dt\\
 &X_0=x
\end{split}
\end{equation}
with $\sigma=(\sigma_{ij})$, $i=1,\cdots,d$, $j=1,\cdots,r$ and $b=(b_1,\cdots,b_d)$, where $\sigma_{ij}$ and $b_i$ are given $C^\infty$ functions on  $\mathbb{R}^d$ with bounded derivatives. In particular, we have \begin{align*}
\|\sigma(x)\|+\|b(x)\|:=\big(\sum_{i=1}^d\sum_{j=1}^r|\sigma_{ij}(x)|^2\big)^{1/2}+\big(\sum_{i=1}^d|b_i(x)|^2\big)^{1/2}\leq K(1+|x|)
\end{align*}
for some $K>0$ and $|x|^2:=\sum_{i=1}^dx_i^2$. Under the above assumptions on $\sigma$ and $b$, it is well known that a unique, non-explosive strong
solution $\{X(t,x,\omega)\}_{t\geq0,x\in\mathbb{R}^d}$ exists on $(\Omega,\mathcal{F},P)$ (see ref. \cite{MR1011252}).
\begin{theorem}[{\cite{MR1472487}}]\label{thm1}
For $x\in\mathbb{R}^d$ and $t\geq0$, let $\{X(t,x,\omega)\}$ be the unique strong solution of \eqref{eq1}. Then there exists a process
$\{\tilde{X}(t,x,\omega)\}_{t\geq0,x\in\mathbb{R}^d}$ such that
\begin{enumerate}[label=(\roman*)]
\item For all $x\in\mathbb{R}^d$, $P\{\tilde{X}(t,x,\omega)=X(t,x,\omega),\ \forall t\geq0\}=1$.
\item For a.e. $\omega(P)$, $x\rightarrow\tilde{X}(t,x,\omega)$ is $C^\infty$ diffeomorphism for all $t\geq0$.
\item Let $\theta_t:\Omega\rightarrow\Omega$ be the shift operator i.e. $\theta_t\omega(s)=\omega(s+t)$; then for $s,t\geq0$, we have
\begin{align*}
\tilde{X}(t+s,x,\omega)=\tilde{X}(s,\tilde{X}(t,x,\omega),\theta_t\omega)
\end{align*}
for all $x\in\mathbb{R}^d$, a.e. $\omega(P)$.
\end{enumerate}
\end{theorem}
We denote the modification obtained in Theorem \ref{thm1} again by $\{X(t,x,\omega)\}$. For $\omega$ outside a null set
$\tilde{N}$, the flow of diffeomorphisms induces, for each $t\geq0$ a continuous linear map, denoted by $X_t(\omega)$ on $C^{\infty}$.
The map $X_t(\omega):C^{\infty}\rightarrow C^{\infty}$ is given by $X_t(\omega)(\varphi)(x)=\varphi(X(t,x,\omega))$. This map is linear and continuous w.r.t. the
topology on $C^{\infty}$ given by the following family of seminorms: For $K\subset\mathbb{R}^d$ a compact set, let
$\|\varphi\|_{n,K}:=\max_{|\alpha|\leq n}\sup_{x\in K}|D^{\alpha}\varphi(x)|$ where $\varphi\in C^{\infty}$ and $n\geq1$ an integer and
$\alpha=(\alpha_1,\alpha_2,\cdots,\alpha_d)$ and $|\alpha|=\alpha_1+\alpha_2+\cdots+\alpha_d$. Let $K_{t,\omega}$ denote the image of $K$ under the map
$x\rightarrow X(t,x,\omega)$. Then using the chain rule we can show that there exists a constant $C(t,\omega)>0$ such that
\begin{align*}
\|X_t(\omega)(\varphi)\|_{n,K}\leq C(t,\omega)\|\varphi\|_{n,K_{t,\omega}}.
\end{align*}
Let $X_t(\omega)^\ast$ denote the transpose of the map
$X_t(\omega):C^{\infty}\rightarrow C^{\infty}$. Let
$\mathcal{E}^\prime$ denote the space of distributions with compact
support. We will denote the duality between $\mathcal{E}^{\prime}$
and $C^{\infty}$ by $\langle \cdot, \cdot\rangle$. Below we will use
the same notation for the $\mathcal{L}^2$-inner product. Then
$X_t(\omega)^\ast:\mathcal{E}^\prime\rightarrow\mathcal{E}^\prime$
is given by
\begin{align*}
\langle X_t(\omega)^\ast\psi,\varphi\rangle=\langle\psi,X_t(\omega)\varphi\rangle
\end{align*}
for all $\varphi\in C^{\infty}$ and $\psi\in\mathcal{E}^\prime$. For
subsets $K$ of ${\mathbb R}^d$, we will denote by
$\mathcal{E}^\prime(K)$ the set of $\psi\in\mathcal{E}^\prime$ with
$supp\ \psi\subseteq K$. Let $K$ be a compact subset of ${\mathbb
R}^d$ and let $\psi\in\mathcal{E}^\prime(K)$. Let
$N=order(\psi)+2d$. Then there exist continuous functions
$g_{\alpha}$, $|\alpha|\leq N$, $supp\ g_{\alpha}\subseteq V$ where
$V$ is an open set having compact closure, containing $K$, such that
\begin{equation}\label{eq2}
\psi=\sum_{|\alpha|\leq N}\partial^{\alpha}g_{\alpha}.
\end{equation}
See \cite{MR2296978}. Let $\varphi\in C^{\infty}$. Let $f_i\in
C^{\infty}$ and $f=(f_1,\cdots,f_d)$. Let $\alpha$ be a multi index.
We now describe each of the numbers $\partial^{\alpha}(\varphi\circ
f)(x)$, $x\in\mathbb{R}^d$ as the result of a distribution
(depending on $x\in\R^d$) acting on the test function $\phi$. Let
$\beta^i$, $i=1,\cdots,d$ be multi indices, each with $d$
components. Using the chain rule for differentiation, we can verify
that for each multi index $\gamma$ with $|\gamma|\leq|\alpha|$,
there exist polynomials $P_{\gamma}$, in a finite number of
variables, with $deg\ P_{\gamma}=|\gamma|$, such that
\begin{equation}\label{eq3}
\partial^{\alpha}(\varphi\circ f)(x)=\sum_{|\gamma|\leq|\alpha|}(-1)^{|\gamma|}
P_{\gamma}((\partial^{\beta^1}f_1,\cdots,\partial^{\beta^d}f_d)_{|\beta^i|\leq|\alpha|})(x)\langle\varphi,\partial^{\gamma}\delta_{f(x)}\rangle.
\end{equation}
For $\omega\notin\tilde{N}$, define $Z_t(\omega):\E^\prime\rightarrow\E^\prime$ by
\begin{equation}\label{eq4}
\begin{split}
Z_t(\omega)(\psi)=&\sum_{|\alpha|\leq N}(-1)^{|\alpha|}\sum_{|\gamma|\leq|\alpha|}(-1)^{|\gamma|}\int_Vg_{\alpha}(x)\\
&P_{\gamma}((\partial^{\beta^1}X_1,\cdots,\partial^{\beta^d}X_d)_{|\beta^i|\leq|\alpha|})(t,x,\omega)\partial^{\gamma}\delta_{X(t,x,\omega)}dx
\end{split}
\end{equation}
Take $Z_t(\omega)=0$ if $\omega\in\tilde{N}$.

Let $\Sc$ be the space of smooth rapidly decreasing functions on $\R^d$ with dual $\Sc^\prime$, the space of tempered distributions (see \cite{MR771478}). For $p \in \R$, consider the increasing norms $\|\cdot\|_p$, defined by the inner
products
\[\langle f,g\rangle_p:=\sum_{|k|=0}^{\infty}(2|k|+d)^{2p}\langle f,h_k\rangle\langle g,h_k\rangle,\ \ \ f,g\in\Sc.\]
Here, $\{h_k\}_{|k|=0}^{\infty}$ is an orthonormal basis for $\mathcal{L}^2(\R^d,dx)$ given by Hermite functions (for $d=1$,
$h_k(t)=(2^kk!\sqrt{\pi})^{-1/2}\exp\{-t^2/2\}H_k(t)$, where $H_k$ are the Hermite polynomials, see \cite{MR771478}), $\langle\cdot,\cdot\rangle$ is the usual
inner product in $\mathcal{L}^2(\R^d,dx)$. We define the Hermite-Sobolev spaces $\Sc_p, p \in \R$ as the completion of $\Sc$ in
$\|\cdot\|_p$. Note that the dual space $\Sc_p^\prime$ is isometrically isomorphic with $\Sc_{-p}$ for $p\geq 0$. We also have $\Sc = \bigcap_{p}(\Sc_p,\|\cdot\|_p), \Sc^\prime=\bigcup_{p>0}(\Sc_{-p},\|\cdot\|_{-p})$ and $\Sc_0 = \mathcal{L}^2(\R^d)$. The space $C^\infty_c(\R^d)$ of smooth functions with compact support is dense in $\Sc$ (in $\|\cdot\|_p$) and hence in $\Sc_p$, for $p \in \R$.

\begin{theorem}[{\cite{MR2373102}}]\label{thm2}
Let $\psi$ be a distribution with compact support having representation \eqref{eq2}. Let $p>0$ be such that $\partial^{\alpha}\delta_x\in \Sc_{-p}$
for $|\alpha|\leq N$. Then $\{Z_t(\psi)\}_{t\geq0}$ is an $\Sc_{-p}$ valued continuous adapted process such that for all $t\geq0$,
\begin{equation}\label{Z-as-dual-X}
Z_t(\psi)=X_t^\ast(\psi)\ \ a.s.\ P.
\end{equation}

\end{theorem}

\begin{example}\label{spl-example}
We mention two examples corresponding to special initial values
$\psi$ for which the process $\{Z_t(\psi)\}$ is the solution of the
SPDE \eqref{SPDE}. Uniqueness of the solution in the case of the
first example is one of the principal motivations for and the main
application of, the results of this paper. We refer to the results
of \cite{ MR3063763} for uniqueness in the case of the second
example.
\begin{enumerate}
\item Let $\psi \in C^\infty_c(\R^d)$. Then $Z_t(\psi) = \int_{\R^d}\psi(x)\delta_{X(t,x)}\, dx$. This fact can be verified as follows.
\[\langle Z_t(\psi), \phi \rangle = \int_{\R^d}\psi(x)\phi(X(t,x))\, dx = \int_{\R^d}\psi(x)(X_t(\phi))(x)\, dx = \langle \psi, X_t(\phi)\rangle.\]
Moreover, $Z_t(\psi)$ is actually a function. To see this, let $J(t,x)$ denote the Jacobian obtained by the change of variables $x$ to $X(t,x)$. Since $x \mapsto X(t,x)$ is a diffeomorphism, $J(t,x)$ is non-zero, and in particular, $J(t,x)$ is either strictly positive or strictly negative. Now
\[\langle Z_t(\psi), \phi \rangle = \int_{\R^d}\psi(X(t,\cdot)^{-1}x)\phi(x)|J(t,x)|\, dx.\]
Therefore $Z_t(\psi)$ is given by the $C^\infty_c(\R^d)$ function $x
\mapsto \psi(X(t,\cdot)^{-1}x)|J(t,x)|$. Note that the same
computations go through if $\psi \in \Sc$. However,
\eqref{Z-as-dual-X} need not hold.
\item Take $\psi = \delta_x$ for some  $x \in \R^d$. Then  $Z_t(\psi) = \delta_{X(t,x)}$.
\end{enumerate}

\end{example}

We now define the operators $A:C^{\infty}\rightarrow\mathcal{L}(\R^r,C^{\infty})$ and $L:C^{\infty}\rightarrow C^{\infty}(\R^d)$ as follows: for
$\varphi\in C^{\infty}$, $x\in\R^d$,
\begin{align*}
&A\varphi=(A_1\varphi,\cdots,A_r\varphi),\\
&A_i\varphi(x)=\sum_{k=1}^d\sigma_{ki}(x)\partial_k\varphi(x),\\
&L\varphi(x)=\frac{1}{2}\sum_{i,j=1}^d(\sigma\sigma^t)_{ij}(x)\partial^2_{i,j}\varphi(x)+\sum_{i=1}^db_i(x)\partial_i\varphi(x).
\end{align*}
\begin{remark}
Since $\sigma$, $b$ are $C^{\infty}$ functions on $\R^d$ with bounded derivatives satisfying linear growth condition, therefore
$L:\Sc\rightarrow\Sc$.
\end{remark}

We define the adjoint operators $A^\ast:\E^\prime\rightarrow\mathcal{L}(\R^r,\E^\prime)$ and $L^\ast:\E^\prime\rightarrow\E^\prime$
\begin{align*}
&A^\ast\psi=(A^\ast_1\psi,\cdots,A^\ast_r\psi),\\
&A^\ast_i\psi=-\sum_{k=1}^d\partial_k(\sigma_{ki}\psi),\\
&L^\ast\psi=\frac{1}{2}\sum_{i,j=1}^d\partial^2_{i,j}((\sigma\sigma^t)_{ij}\psi)-\sum_{i=1}^d\partial_i(b_i\psi).
\end{align*}

\begin{proposition}[{\cite[Proposition 3.2]{MR2373102}}]\label{prop1}
Let $\sigma_{ij}$, $i=1,\cdots,d$, $j=1,\cdots,r$ and $b_1,\cdots,b_d$ be $C^{\infty}$ functions on $\R^d$ with bounded derivatives. Let $p>0$ and $q>[p]+4$, where $[p]$ denotes the largest integer less than or equal to $p$. Let $K$ be a compact subset of
$\R^d$. Then, $A^\ast:\Sc_{-p}\cap\E^\prime(K)\rightarrow\mathcal{L}(\R^r,\Sc_{-q}\cap\E^\prime(K))$ and $L^\ast:\Sc_{-p}\cap\E^\prime(K)\rightarrow \Sc_{-q}\cap\E^\prime(K)$. Moreover, there exists
constants $C_1(p)>0$, $C_2(p)>0$ independent of the compact set $K$ such that
\[\|A^\ast\psi\|_{HS(-q)}\leq C_1(p)\|\psi\|_{-p},\ \ \ \ \|L^\ast\psi\|_{-q}\leq C_2(p)\|\psi\|_{-p}\]
where
\[\|A^\ast\psi\|^2_{HS(-q)}:=\sum_{i=1}^r\left\|\sum_{k=1}^d\partial_k(\sigma_{ki}\psi)\right\|_{-q}^2=\sum_{i=1}^r\|A_i^\ast\psi\|^2_{-q}.\]
\end{proposition}

\begin{definition}\label{loc-cpt-supp}
We say that an $\E^\prime$ valued process $\{Y_t\}$ is locally of compact support if there exists an increasing
sequence of stopping times $\{\tau_n\}$ such that $\tau_n \uparrow \infty$ and for each $n$, a.s.
$supp(Y^{\tau_n}_t) \subset K_n, \forall t$ where $\{K_n\}$ is some increasing family of compact sets.
\end{definition}

\begin{proposition}\label{cpt-supp-st-intg}
Let $
Y = (Y^1,\cdots,Y^r)$, where each $\{Y_t^i\}$ is an $\Sc_{-p}\cap\E^\prime$ valued adapted process with continuous paths in $\|\cdot\|_{-p}$ norm and is locally of compact support. Then the local martingale $\{\int_0^t Y_s\cdot dB_s\}$ is locally of compact support.
\end{proposition}

\begin{proof}
By our hypothesis, there exists an increasing
sequence of stopping times $\{\tau_n\}$ such that $\tau_n \uparrow \infty$ and for each $n$, a.s.
$supp((Y^i)^{\tau_n}_t) \subset K_n, \forall t, \forall i$, where $\{K_n\}$ is some increasing family of compact sets. Since $\{Y_t\}$ has continuous paths, without loss of generality, we assume that $\|Y^{\tau_n}_t\|_{HS(-p)} \leq n$ for $t \geq 0$. Hence we have the existence of the stochastic integral $\{\int_0^{t\wedge\tau_n} Y_s\cdot dB_s\}$.\\
Suppose $\phi$ is a $C^{\infty}$ function such that the support of $\phi$ and its derivatives are contained in the complement of $K_n$. Then, a.s. for $t \geq 0$
\[\left\langle \int_0^{t\wedge\tau_n} Y_s\cdot dB_s, \phi\right\rangle = \sum_{i=1}^r\int_0^{t\wedge\tau_n} \langle (Y^i)^{\tau_n}_s,\phi\rangle dB_s = 0.\]
Since, by definition, $\int_0^t Y_s\cdot dB_s = \int_0^{t\wedge\tau_n} Y_s\cdot dB_s$, for $t \leq \tau_n$, the result follows.
\end{proof}

The open set $V$ mentioned before \eqref{eq2} is bounded. Hence there exists $\lambda > 0$ such that $V$ is a subset of the closed ball $\overline{B(0,\lambda)}$ of radius $\lambda$ centered at the origin. For $R \geq 0$, define \[\tau_R:=\inf\{t>0: \sup_{s \in [0,t]}\sup_{|x|\leq\lambda}|X_s(x)|\geq R\}.\]
Since $\{\sup_{|x|\leq\lambda}|X_s(x)|\}_s$ is an adapted process, the process $\{\sup_{s \in [0,t]}\sup_{|x|\leq\lambda}|X_s(x)|\}_t$ is adapted and increasing. Hence $\tau_R$ is a stopping time for each $R$.

\begin{proposition}\label{specific-cpt-support}
The process $\{Z_t(\psi)\}$, defined by \eqref{eq4}, is locally of compact support. Furthermore,
\begin{enumerate}
\item $t\leq \tau_R\Rightarrow supp(Z_t(\psi))\subseteq\overline{B(0,R)}$.
\item As $R\uparrow\infty$, $\tau_R\uparrow \infty$.
\end{enumerate}
\end{proposition}

\begin{proof}
Suppose $\varphi$ is a $C^{\infty}$ function such that the support
of $\varphi$ is contained in the
 complement of $\overline{B(0,R)}$. Then \eqref{eq4} and $t\leq \tau_R$ imply
\begin{align*}
\langle Z_t(\psi), \varphi\rangle=&\sum_{|\alpha|\leq N}(-1)^{|\alpha|}\sum_{|\gamma|\leq|\alpha|}(-1)^{|\gamma|}\int_Vg_{\alpha}(x)\\
&P_{\gamma}((\partial^{\beta^1}X_1,\cdots,\partial^{\beta^d}X_d)_{|\beta^i|\leq|\alpha|})(t,x)\left\langle\partial^{\gamma}\delta_{X(t,x)},\varphi\right\rangle dx\\
=&\sum_{|\alpha|\leq N}(-1)^{|\alpha|}\sum_{|\gamma|\leq|\alpha|}(-1)^{|\gamma|}\int_Vg_{\alpha}(x)\\
&P_{\gamma}((\partial^{\beta^1}X_1,\cdots,\partial^{\beta^d}X_d)_{|\beta^i|\leq|\alpha|})(t,x)(-1)^{|\gamma|}\left\langle\delta_{X(t,x)},\partial^{\gamma}\varphi\right\rangle dx\\
= & 0.
\end{align*}
Hence $supp(Z_t(\psi))\subseteq\overline{B(0,R)}$.\\
Since the process $\{\sup_{s \in [0,t]}\sup_{|x|\leq\lambda}|X_s(x)|\}$ is finite for all $t$, $\tau_R \uparrow \infty$ as $R \uparrow \infty$.
\end{proof}

We consider the following stochastic partial differential equation
in $\E^\prime$,
\begin{equation}\label{eq5}
\begin{split}
&dY_t=L^\ast Y_tdt+A^\ast Y_t\cdot dB_t\\
&Y_0=\psi.
\end{split}
\end{equation}
Here the term $\int A^\ast Y_t\cdot dB_t$ denotes the expression $\sum_{i=1}^r \int A^\ast_iY_t\, dB_t^i$, where $B^i, i=1,\cdots,r$ are the components of the Brownian motion $\{B_t\}$.

\begin{definition}\label{strong}
Let $p \in \R$ and $\psi \in \Sc_p\cap \E^\prime$. Let $q$ be such that $A^\ast:\Sc_p\cap\E^\prime(K)\rightarrow\mathcal{L}(\R^r,\Sc_q\cap\E^\prime(K))$ and $L^\ast:\Sc_p\cap\E^\prime(K)\rightarrow \Sc_q\cap\E^\prime(K)$ are bounded linear operators for each compact set $K$ in $\R^d$. We say that $\{Y_t\}$ is a $(p,q)$ strong solution of \eqref{eq5} if it is an $\Sc_p\cap\E^\prime$-valued adapted process, has continuous paths in $\Sc_p$, is locally of compact support and satisfies the following equation
in $\Sc_q$, a.s.,
\begin{equation}\label{eq6}
Y_t=\psi+\int_0^tA^\ast Y_s\cdot dB_s+\int_0^tL^\ast Y_sds
\end{equation}
for all $t\geq0$.
\end{definition}
\begin{remark}\label{loc-cpt-supp-leb} We note that the process in the third term in the right hand side of
equation \eqref{eq6} is also locally of compact support. The proof
is the same as in the proof of Proposition \ref{cpt-supp-st-intg}.
\end{remark}

\begin{theorem}[{\cite{MR2373102}}]\label{thm3}
Let $\psi\in\E^\prime$ have the representation \eqref{eq2}. Let $p>0$ be such that $\partial^{\gamma}\delta_x\in \Sc_{-p}$, $|\gamma|\leq N$. Let $q>p$
be as in Proposition \ref{prop1}. Then the $\Sc_{-p}$-valued continuous, adapted process $\{Z_t(\psi)\}_{t\geq0}$
defined by \eqref{eq4}, is a $(-p,-q)$ strong solution of \eqref{eq5}.
\end{theorem}
\begin{remark} It was noted in \cite[Theorem 4.1]{MR2373102} that $p
> \tfrac{d}{4}+\tfrac{|\gamma|}{2}$ is a sufficient condition for
$\partial^{\gamma}\delta_x\in \Sc_{-p}$, for any multi-index
$\gamma$. Thus in the previous theorem, we can state an explicit
condition on $p$ as $p > \tfrac{d}{4}+\tfrac{N}{2}$ .
\end{remark}
\begin{proposition}[{\cite{MR2373102}}]\label{prp2}
Let $\psi\in\E^\prime$ with representation \eqref{eq2}. Let $p>\frac{d}{4}+\frac{N}{2}$ where $N=order(\psi)+2d$. Let $\{Z_t(\psi)\}$ be the $\Sc_{-p}$ valued
continuous adapted process defined by \eqref{eq4}. Then for all $T>0$,
\[\sup_{t\leq T}\Exp\|Z_t(\psi)\|^2_{-p}<\infty. \]
\end{proposition}

\section{Mild solutions}\label{s:3}
Let $\{S_t\}_{t\geq0}$ be the semigroup corresponding to
$\{X(t,x)\}$ solving \eqref{eq1} i.e. for $f\in\Sc$, $S_tf(x):=\Exp f(X(t,x))$. Then,
\begin{align*}
\mathopen{_{C^{\infty}}\langle}S_tf,\psi\rangle_{\E^\prime}&=\langle\Exp f\circ X_t,\psi\rangle=\Exp\langle X_t(f),\psi\rangle\\
&=\Exp\langle f,Z_t(\psi)\rangle=\mathopen{_{\Sc}\langle}f,\Exp Z_t(\psi)\rangle_{\Sc^\prime}.
\end{align*}

Consider the map $S_t^\ast :\E^\prime\rightarrow \Sc^\prime$ defined by $S_t^\ast \psi :=\Exp Z_t(\psi)$.
\begin{theorem}[{\cite{MR2373102}}]\label{thm4}
The following are the properties of the operators $S_t$ and $S_t^\ast$.
\begin{enumerate}[label=\alph*)]
\item We have $S_t:\Sc\rightarrow C^{\infty}$. The map $S_t^\ast :\E^\prime\rightarrow \Sc^\prime$ is adjoint to $S_t$ in the sense that
\[\mathopen{_{\Sc^\prime}\langle}S_t^\ast \psi,\phi\rangle_{\Sc} =\mathopen{_{\E^\prime}\langle}\psi,S_t\phi\rangle_{C^\infty}\]
for all $\psi\in\E^\prime$ and $\phi\in\Sc$.
\item Let $K\subset\R^d$ be a compact set and $p>0$. Then for $q>\frac{5}{4}d+[p]+1$, $S_t^\ast :\Sc_{-p}\cap\E^\prime(K)\rightarrow \Sc_{-q}$ is a bounded linear operator.
Further, for any $T>0$, there exists a constant $C(T), 0 < C(T) <
\infty$ such that
\[\sup_{t\leq T}\|S_t^\ast \|_H<C(T)\]
where $\|\cdot\|_H$ is the operator norm on the Banach space $H$ of bounded linear operators from $\Sc_{-p}\cap\E^\prime(K)$ to $\Sc_{-q}$.
\end{enumerate}
\end{theorem}

As a consequence of Proposition \ref{cpt-supp-st-intg}, we get the next result.

\begin{corollary}
Let $
Y = (Y^1,\cdots,Y^r)$, where each $\{Y_t^i\}$ is an $\Sc_{-p}\cap\E^\prime$ valued adapted process with continuous paths in $\Sc_{-p}$ and is locally of compact support. Let $q$ be as in Theorem \ref{thm4}. Then for each $i$, $\{S^\ast _{t-s}Y_s^i\}_{s \in [0,t]}$ is an $\Sc_{-q}$ valued continuous adapted process and the process $\{\int_0^tS^\ast _{t-s}Y_s\cdot dB_s\}$ is an $\Sc_{-q}$ valued continuous local martingale. Here the term $\int_0^t S^\ast _{t-s} Y_s\cdot dB_s$ denotes the sum $\sum_{i=1}^r \int_0^t S^\ast _{t-s} Y^i_s dB_s^i$, where $B^i, i=1,\cdots,r$ are the components of the Brownian motion $\{B_t\}$.
\end{corollary}

In what follows, $p$ will denote an arbitrary but fixed non-negative real number. We also associate two positive real numbers $p^\prime$ and $q$ to this $p$. By Proposition \ref{prop1}, we can choose $p^\prime>[p]+4$ such that
\begin{equation}\label{L*}
L^\ast :\Sc_{-p}\cap\E^\prime\big(\overline{B(0,R)}\big)\rightarrow \Sc_{-p^\prime}\cap\E^\prime\big(\overline{B(0,R)}\big)
\end{equation}
and
\begin{equation}\label{A*}
A^\ast :\Sc_{-p}\cap\E^\prime\big(\overline{B(0,R)}\big)\rightarrow\mathcal{L}(\R^r,\Sc_{-p^\prime}\cap\E^\prime\big(\overline{B(0,R)}\big))
\end{equation}
are bounded linear operators for any $R > 0$. Now, by Theorem \ref{thm4}, we can choose $q>\frac{5}{4}d+[p^\prime]+1$ such that
\begin{equation}\label{S*}
S_t^\ast :\Sc_{-p^\prime}\cap\E^\prime\big(\overline{B(0,R)}\big)\rightarrow \Sc_{-q}
\end{equation}
is a bounded linear operator for any $R > 0$. Note that $0<p<p^\prime<q$.

\begin{lemma}\label{lema1}
For $x\in \Sc_{-p}\cap\E^\prime$
\begin{enumerate}[label=(\roman*)]
\item
\[L^\ast x=\lim_{t\rightarrow0+}\frac{S_t^\ast x-x}{t},\]
\item
\[\frac{d}{dt}S_t^\ast x=S_t^\ast L^\ast x, \]
\item
\[S_t^\ast x-S_s^\ast x=\int_s^tS_u^\ast L^\ast xdu.\]
\end{enumerate}
\end{lemma}

\begin{proof}
The proof follows from standard duality arguments.
\end{proof}

\begin{definition}\label{defn1}
Let $\psi \in \Sc_{-p}\cap \E^\prime$. We say that $\{Y_t\}$ is a $(-p,-q)$ mild solution of \eqref{eq5} if it is an $\Sc_{-p}\cap\E^\prime$-valued  adapted process, with continuous paths in $\Sc_{-p}$ and is locally of compact support and satisfies the following equation
in $\Sc_{-q}$, a.s.,
\begin{equation}\label{eq7}
Y_t=S_t^\ast \psi+\int_0^t S^\ast _{t-s}A^\ast Y_s\cdot dB_s
\end{equation}
for all $t\geq 0$. \end{definition}

\begin{remark}
As mentioned in Remark \ref{loc-cpt-supp-leb}, if $\{Y_t\}$ is a
strong solution, then all the terms in \eqref{eq6} are locally of
compact support (see Proposition \ref{cpt-supp-st-intg}). For an
arbitrary $\psi \in \E^\prime$, the distribution $S_t^\ast \psi$
need not be compactly supported. To see this, take $r=d, b(\cdot)
\equiv 0, \psi = \delta_x$ and $\sigma(\cdot) \equiv I_d$, the
identity matrix. Then $Z_t(\psi) = \delta_{X(t,x)} =
\delta_{B_t+x}$. Then $S^\ast_t\psi$ is not compactly supported. As
such, if $\{Y_t\}$ is a mild solution, then the terms on the right
hand side of \eqref{eq7} need not be locally of compact support.
\end{remark}

\begin{proposition}\label{prop-nrm}
For each $R>0$, the map $t \mapsto S_t^\ast $ is of finite variation in the operator norm $\|\cdot\|_{\mathcal{L}(\Sc_{-p}\cap\E^\prime\big(\overline{B(0,R)}\big),\Sc_{-q})}$. In particular, For all $x\in \Sc_{-p}\cap\E^\prime$, the map $t \mapsto S_t^\ast x$ is of finite variation in $\|\cdot\|_{-q}$ norm.
\end{proposition}

\begin{proof}
Let $\Pi=\{0=t_0<t_1<t_2\cdots<t_n=t\}$ be a partition of $[0,t]$. Observe that \begin{align*}
&\sum_{i=0}^{n-1}\left\|S^\ast _{t_{i+1}}-S^\ast _{t_i}\right\|_{\mathcal{L}(\Sc_{-p}\cap\E^\prime\big(\overline{B(0,R)}\big),\Sc_{-q})}\\
&=\sum_{i=0}^{n-1}\sup_{\substack{\|x\|_{-p}\leq1,\\ x \in \E^\prime\big(\overline{B(0,R)}\big)}}\|S^\ast _{t_{i+1}}x-S^\ast _{t_i}x\|_{-q}\\
&=\sum_{i=0}^{n-1}\sup_{\substack{\|x\|_{-p}\leq1,\\ x \in \E^\prime\big(\overline{B(0,R)}\big)}}\left\|\int_{t_i}^{t_{i+1}}S_s^\ast  L^\ast  xds\right\|_{-q}\\
&\leq\sum_{i=0}^{n-1}\sup_{\substack{\|x\|_{-p}\leq1,\\ x \in \E^\prime\big(\overline{B(0,R)}\big)}}\int_{t_i}^{t_{i+1}}\|S_s^\ast  L^\ast  x\|_{-q}ds\\
&\leq\sum_{i=0}^{n-1}\sup_{\substack{\|x\|_{-p}\leq1,\\ x \in \E^\prime\big(\overline{B(0,R)}\big)}}\int_{t_i}^{t_{i+1}}C(T)\|L^\ast x\|_{-p^\prime}ds\ \ \text{[by b) of Theorem \ref{thm4}]}\\
&\leq\sum_{i=0}^{n-1}\sup_{\substack{\|x\|_{-p}\leq1,\\ x \in \E^\prime\big(\overline{B(0,R)}\big)}}\int_{t_i}^{t_{i+1}}C(T)C_2(p)\|x\|_{-p}ds\ \ [\text{by Proposition}\ \ref{prop1}]\\
&\leq\sum_{i=0}^{n-1}\sup_{\substack{\|x\|_{-p}\leq1,\\ x \in \E^\prime\big(\overline{B(0,R)}\big)}}C(T)C_2(p)(t_{i+1}-t_i)\|x\|_{-p}\\
&\leq C(T)C_2(p)\sum_{i=0}^{n-1}(t_{i+1}-t_i)\\
&=C(T)C_2(p) t.
\end{align*}
Hence the proof.
\end{proof}

Let $R>0$. Let $t \in [0,\infty)\mapsto y_t$ be an
$\Sc_{-p}\cap\E^\prime\big(\overline{B(0,R)}\big)$ valued continuous
map. Let $\Pi_m = \{0=t^m_0<t^m_1<t^m_2\cdots<t^m_n=t\}$ be a
sequence of partitions of $[0,t]$ such that $|\Pi_m| = \max_i
|t_{i+1}^m -t_i^m| \to 0$ as $m \to \infty$. Let us consider the
simple functions $s \mapsto
y^m_s:=\sum_{i=0}^{n-1}\mathbbm{1}_{(t^m_i,t^m_{i+1}]}(s)
y_{t_i^m}$. Define
\begin{equation}\label{intg-dfn}
\int_0^tdS_s^\ast y_s := \lim_{m \to \infty} \int_0^tdS_s^\ast y^m_s
:=\lim_{m \to \infty}\sum_{i=0}^{n-1}[S^\ast _{t^m_{i+1}}-S^\ast
_{t^m_i}][y_{t^m_i}]
\end{equation}

\begin{proposition}\label{prop-int}
The limit in \eqref{intg-dfn} exists as an element of $\Sc_{-q}$ and is independent of the sequence of partitions chosen.
\end{proposition}
\begin{proof}
Fix $T > 0$. The map $s \mapsto y_s$ is uniformly continuous on $[0,T]$. Therefore given any $\epsilon > 0$, there exists $\delta > 0$ such that $\|y_u - y_v\|_{-p} < \epsilon$ whenever  $u,v \in [0,T]$ with $|u-v| < \delta$.\\
Choose $m,l$ sufficiently large such that $|\Pi_m| < \delta, |\Pi_l| < \delta$. Let us denote
\[h^m:=\int_0^tdS_s^\ast y^m_s\ \ \text{and}\ \ h^l:=\int_0^tdS^\ast _sy^l_s.\]
Let $\Pi:=\{0=r_0<r_1<r_2\cdots<r_k=t\}$ be the refinement of the two partitions. Note that $|\Pi|< \delta$. In particular $\|y^m_{r_i}-y^l_{r_i}\|_{-p} < \epsilon$. Now, we show that $\|h^m-h^l\|_{-q}\rightarrow0$
as $m,l\rightarrow\infty$.
\begin{align*}
\|h^m-h^l\|_{-q}&=\left\|\int_0^tdS_s^\ast [y^m_s-y^l_s]\right\|_{-q}\\
&=\left\|\sum_{i=0}^{k-1}[S^\ast _{r_{i+1}}-S^\ast _{r_i}][y^m_{r_i}-y^l_{r_i}]\right\|_{-q}\\
&\leq\sum_{i=0}^{k-1}\|S^\ast _{r_{i+1}}-S^\ast _{r_i}\|_{\mathcal{L}(\Sc_{-p}\cap\E^\prime\big(\overline{B(0,R)}\big),\Sc_{-q})}\|y^m_{r_i}-y^l_{r_i}\|_{-p}\\
&\leq\epsilon\sum_{i=0}^{k-1}\|S^\ast _{r_{i+1}}-S^\ast _{r_i}\|_{\mathcal{L}(\Sc_{-p}\cap\E^\prime\big(\overline{B(0,R)}\big),\Sc_{-q})}\\
&\leq \epsilon \ TV_{[0,t]}(S^\ast_\cdot),
\end{align*}
where $TV_{[0,t]}(S^\ast_\cdot)$ denotes the total variation of the map $s \mapsto S^\ast_s$ on $[0,t]$, which is finite from Proposition \ref{prop-nrm}. Since $\epsilon$ was arbitrary, the sequence $\{h^m\}$ is Cauchy and hence $\lim_{m\to \infty}h^m$ exists.\\
By standard arguments, we can show the limit is independent of the sequence of partitions chosen.
\end{proof}

\begin{theorem}\label{stron-mild}
Let $L^\ast $, $A^\ast $ and $S_t^\ast $ satisfy \eqref{L*}, \eqref{A*} and \eqref{S*} respectively.
Let $\{Y_t\}$ be a $(-p,-q)$ strong solution of \eqref{eq5}. Then it is also a $(-p,-q)$ mild solution.

\begin{proof}[First proof of Theorem \ref{stron-mild}]
Since $\{Y_t\}$ is locally of compact support, there exists an increasing
sequence of stopping times $\{\tau_n\}$ such that $\tau_n \uparrow \infty$ and for each $n$, a.s.
$supp(Y^{\tau_n}_t) \subset \overline{B(0,R_n)}, \forall t$ for some $R_n > 0$.

Fix $t \geq 0$ and $x \in \Sc$. Fix a natural number $n$. Consider the function $G:[0,t]\times(\Sc_{-p^\prime}\cap\E^\prime(\overline{B(0,R_n)}))\rightarrow\R$, where $G\in C^{1,2}([0,t]\times(\Sc_{-p^\prime}\cap\E^\prime(\overline{B(0,R_n)})))$ defined by
\[G(s,y):=\mathopen{_{\Sc_{q}}\langle} x, S_{t-s}^\ast y\rangle_{\Sc_{-q}}.\]
Recall that by Theorem \ref{thm4}, $\sup_{s \in [0,t]}\|S_s^\ast \|_H<C(t)$
where $\|\cdot\|_H$ is the operator norm on the Banach space $H$ of bounded linear operators from $\Sc_{-p^\prime}\cap\E^\prime(\overline{B(0,R_n)})$ to $\Sc_{-q}$ (here $p^\prime,q$ are as in \eqref{S*}). We will consider the function $G(s,Y_s)$ and localization (on $Y$) is necessary to make use of boundedness of the operators $S_t^\ast$.

For $u \in [0,t], y\in \Sc_{-p}\cap\E^\prime(\overline{B(0,R_n)})$, using Lemma \ref{lema1} we have
\[ \langle x, -\int_0^u S_{t-s}^\ast L^\ast y \,ds\rangle  =  \langle  x,-\int^t_{t-u}S_r^\ast L^\ast y\, dr\rangle = \langle x, S_{t-u}^\ast y - S_t^\ast y\rangle = G(u,y)-G(0,y).\]
Then the partial derivative $G_s(s,y)=-\mathopen{_{\Sc_{q}}\langle} x,  S_{t-s}^\ast L^\ast y\rangle_{\Sc_{-q}}, y \in \Sc_{-p}\cap\E^\prime(\overline{B(0,R_n)})$.

Now for $y, z \in \Sc_{-p^\prime}\cap\E^\prime(\overline{B(0,R_n)})$, we have
\[G(s,y+z) - G(s,y) = \langle x, S_{t-s}^\ast z\rangle.\]
Since $z \mapsto \langle x, S_{t-s}^\ast z\rangle$ is a bounded linear functional on $\Sc_{-p^\prime}$, we get the Fr\'echet derivative $G_y(s,y)=\langle x, S_{t-s}^\ast \cdot\rangle:\Sc_{-p^\prime}\to\R$. Consequently, $G_{yy}(s,y)=0$.

Since
\[Y_s^{\tau_n}=\psi+\int_0^{s\wedge\tau_n} A^\ast Y_r^{\tau_n}\cdot dB_r+\int_0^{s\wedge\tau_n} L^\ast Y_r^{\tau_n}dr,\forall s \in [0,t]\]
applying It\^o's formula (see \cite{MR1207136, MR2560625}) we get
a.s., for all $t \geq 0$,
\begin{align*}
\langle x, Y_t^{\tau_n}\rangle - \langle x, S_t^\ast\psi\rangle &= G(t,Y_t^{\tau_n})-G(0,Y^{\tau_n}_0)\\
&= \int_0^{t\wedge\tau_n}  G_s(s,Y_s^{\tau_n})ds+\int_0^{t\wedge\tau_n} G_y(s,Y_s^{\tau_n})L^\ast Y_s^{\tau_n} ds\\
&+\int_0^{t\wedge\tau_n} G_y(s,Y_s^{\tau_n})A^\ast Y_s^{\tau_n}\cdot dB_s\\
&= -\int_0^{t\wedge\tau_n} \langle x,  S_{t-s}^\ast L^\ast Y_s^{\tau_n} \rangle ds+\int_0^{t\wedge\tau_n} \langle x,  S_{t-s}^\ast L^\ast Y_s^{\tau_n} \rangle ds\\
&+\int_0^{t\wedge\tau_n} \langle x,  S_{t-s}^\ast A^\ast Y_s^{\tau_n} \rangle \cdot dB_s\\
&= \langle x,\int_0^{t\wedge\tau_n}   S_{t-s}^\ast A^\ast Y_s^{\tau_n} \cdot dB_s  \rangle.
\end{align*}
Since $x \in \Sc$ was arbitrary, we have a.s., for all $t \geq 0$,
\[Y_t^{\tau_n} -  S_t^\ast\psi = \int_0^{t\wedge\tau_n}   S_{t-s}^\ast A^\ast Y_s^{\tau_n} \cdot dB_s.\]
Letting $n$ go to $\infty$, we get the required relation.
\end{proof}

\begin{proof}[Second proof of Theorem \ref{stron-mild}]
We first claim that \eqref{justifi} holds.
\begin{equation}\label{justifi}
\int_0^tS^\ast _{t-s}dY_s=\int_0^tS^\ast _{t-s}L^\ast Y_sds+\int_0^tS^\ast _{t-s}A^\ast Y_s\cdot dB_s.
\end{equation}
Here
$S_t^\ast :\Sc_{-p^\prime}\cap\E^\prime\big(\overline{B(0,R)}\big)\rightarrow \Sc_{-q}$ is a bounded linear operator for every $R > 0$ and
$\{Y_t\}_{t \geq 0}$ is a continuous semimartingale, which is locally of compact support. The integral on the left hand side of \eqref{intbp1},
i.e. $\int_0^t S_{t-s}^\ast dY_s$ is well defined (see \cite[Chapter 4]{MR688144}).

Since $\{Y_t\}$ is locally of compact support, there exists an increasing
sequence of stopping times $\{\tau_n\}$ such that $\tau_n \uparrow \infty$ and for each $n$, a.s.
$supp(Y^{\tau_n}_t) \subset \overline{B(0,R_n)}, \forall t$ for some $R_n > 0$. Since $\{Y_t\}$ has continuous paths in $\Sc_{-p}$, without loss of generality, we assume $\|Y_t^{\tau_n}\|_{-p}\leq n$.
Let us also consider the partition, $\T:=\{0=s_0<s_1<\cdots<s_n=t\}$,  where $|\T|:=\max_i|s_{i+1}-s_i|$. Now, we can show
\[\sum_{\{s_i\in\T\}}\left\|\int_{s_i}^{s_{i+1}}S^\ast _{t-s_i}L^\ast Y_r^{\tau_k}dr-S^\ast _{t-s_i}L^\ast Y_{s_i}^{\tau_k}(s_{i+1}-s_i)\right\|_{-q} \to 0\]
as $|\T|\rightarrow 0$. Similarly for the stochastic integral,
\[\sum_{\{s_i\in\T\}}\Exp\left\|\int_{s_i}^{s_{i+1}}S^\ast _{t-s_i}A^\ast Y^{\tau_k}_r\cdot dB_r-S^\ast _{t-s_i}A^\ast Y^{\tau_k}_{s_i}\cdot(B_{s_{i+1}}-B_{s_i})\right\|^2_{-q} \to 0\]
as $|\T|\rightarrow 0$. Hence,
\begin{align*}
 &\int_0^{t\wedge\tau_k}S^\ast _{t-s}dY^{\tau_k}_s&\\
 &=\lim_{|\T|\rightarrow0}\sum_{\{s_i\in\T\}}S^\ast _{t-s_i}(Y^{\tau_k}_{s_{i+1}}-Y^{\tau_k}_{s_i})&\\
 &=\lim_{|\T|\rightarrow0}\sum_{\{s_i\in\T\}}\int_{s_i}^{s_{i+1}}S^\ast _{t-s_i}L^\ast Y^{\tau_k}_rdr
              +\lim_{|\T|\rightarrow0}\sum_{\{s_i\in\T\}}\int_{s_i}^{s_{i+1}}S^\ast _{t-s_i}A^\ast Y^{\tau_k}_r\cdot dB_r&\\
&=\lim_{|\T|\rightarrow0}\sum_{\{s_i\in\T\}}S^\ast _{t-s_i}L^\ast Y^{\tau_k}_{s_i}(s_{i+1}-s_i)&\\
&\ \ \ +\lim_{|\T|\rightarrow0}\sum_{\{s_i\in\T\}}S^\ast _{t-s_i}A^\ast Y^{\tau_k}_{s_i}\cdot(B_{s_{i+1}}-B_{s_i}).&\\
&=\int_0^{t\wedge\tau_k}S^\ast _{t-s}L^\ast Y^{\tau_k}_sds+\int_0^{t\wedge\tau_k}S^\ast _{t-s}A^\ast Y^{\tau_k}_s\cdot dB_s.&
\end{align*}
where in the last but one equality, the second limit in the right
hand side is taken in $L^2(\Omega)$. Letting $\tau_k \to \infty$, we
get \eqref{justifi}. Then,
\begin{align}\label{intbp1}
\int_0^tS^\ast _{t-s}A^\ast Y_s\cdot dB_s&=\int_0^tS^\ast _{t-s}dY_s-\int_0^tS^\ast _{t-s}L^\ast Y_sds&\nonumber\\
&=\int_0^tS^\ast _{t-s}dY_s+\int_0^tdS^\ast _{t-s}Y_s.&
\end{align}
Here, we have used the fact that, if $x\in \Sc_{-p}\cap\E^\prime\big(\overline{B(0,R)}\big)$ then $\frac{d}{dt}S_t^\ast x=S_t^\ast L^\ast x$. Note that the second integral on the right hand side of \eqref{intbp1},
$\int_0^tdS^\ast _{t-s}Y_s$ was defined in \ref{intg-dfn}.
\medskip\\
At the end, we show that the cross variation $[S^\ast ,Y]_t =0$ for
$t\geq 0$, where
\[[S^\ast ,Y]_t:=\lim_{\max|t_{i+1}-t_i|\rightarrow0}\sum_{i=0}^{n-1}(S^\ast _{t_{i+1}}-S_{t_i}^\ast )(Y_{t_{i+1}}-Y_{t_i}).\]
This follows from the fact that $\|[S^\ast ,Y^{\tau_n}]_t\|_{-q} =
0$, which can be verified as in Proposition \ref{prop-int}. Now,
from \eqref{intbp1}, we write the following integration by parts
formula
\[\int_0^tS_{t-s}^\ast A^\ast Y_s\cdot dB_{s} =\int_0^tS^\ast _{t-s}dY_s+\int_0^tdS^\ast _{t-s}Y_s =Y_t-S_t^\ast Y_0,
\]
which implies $Y_t=S^\ast _t\psi+\int_0^tS_{t-s}^\ast A^\ast Y_s\cdot dB_{s}$. This completes the proof.
\end{proof}
\end{theorem}

\begin{proposition}\label{mld-mar}[From mild solutions to martingale representations] Fix $x \in \R^d$ and consider the initial condition
$\psi = \delta_x$ in \eqref{eq5}. Then the mild solution
representation \eqref{eq7} is equivalent to the martingale representation
of square integrable functionals of the diffusion $\{X(t,x)\}$.

\begin{proof}
If the martingale representation holds, we have in
particular, for every $f \in \Sc$, the explicit representation (see, for example \cite{MR2590255}),
\begin{equation}\label{1step-mtgle-repn}
f(X(t,x)) = \Exp f(X(t,x)) + \sum_{i=1}^r  \int_0^t A_i S_{t-s}
f(X(s,x)) dB_s^i.
\end{equation}
To see this, consider the function $g:[0,t]\times\R^d\to\R$ given by $g(s,x):=S_{t-s}f(x) = \Exp f(X(t-s,x))$. Then $g \in C^{1,2}([0,t]\times\R^d)$ and It\^{o} formula gives
\begin{align*}
f(X(t,x)) &= g(t,X(t,x))\\
&= S_t f(x) + \sum_{i=1}^r \sum_{j=1}^d \int_0^t \partial_j S_{t-s}f(X(s,x))\sigma_{ji}(X(s,x))~ dB^i_s,
\end{align*}
since $\partial_t g(s,x) + Lg(s,x) = 0,\forall s \in [0,t], x \in
\R^d$. But, from the definition of the operators $A_i$ in Section 2,
\[\sum_{j=1}^d
\partial_j S_{t-s}f(X(s,x))\sigma_{ji}(X(s,x)) = A_i S_{t-s}
f(X(s,x)),\] which implies \eqref{1step-mtgle-repn}.

Since $Z_t(\psi) = \delta_{X(t,x)}$ (see Example \ref{spl-example}) we have by duality

\begin{equation}\label{mild-mtgle}
\langle f,\delta_{X(t,x)}\rangle = \Exp \langle f,
\delta_{X(t,x)}\rangle + \sum_{i=1}^r \int_0^t \langle
f,S_{t-s}^\ast A_i^\ast \delta_{X(s,x)}\rangle dB_s^i.
\end{equation}
Thus \eqref{eq7} holds. Conversely, if \eqref{eq7} holds then the strong
solution $\{Z_t(\psi)\}$, given by \eqref{eq4} is also a mild
solution and $Z_t(\psi) = \delta_{X(t,x)}$ (see Example
\ref{spl-example}). Hence
\begin{align*}
\delta_{X(t,x)} &= S_t^\ast \delta_{x} + \int_0^t S_{t-s}^\ast A^\ast \delta_{X(s,x)}\cdot dB_s\\
&= \Exp \delta_{X(t,x)} + \sum_{i=1}^r \int_0^t S_{t-s}^\ast
A_i^\ast \delta_{X(s,x)} dB_s^i. \end{align*}
 Now for any $f \in \Sc$, we
get from the identity $f(X(t,x)) = \langle f,\delta_{X(t,x)}\rangle$ that \eqref{mild-mtgle} holds.

Using the above representation and Markov property, one can get
martingale representations for functionals of the form
$f_1(X(t_1,\cdot))f_2(X(t_2,\cdot))\cdots f_k(X(t_k,\cdot))$, where
$f_i \in \Sc$, as in \cite{MR2569262,MR2590255}. Using density arguments, one
can get representations for all square integrable functionals of the diffusion $\{X(t,x)\}$.
\end{proof}
\end{proposition}

\begin{theorem}\label{mild-strong}
Let $\{Y_t\}$ be a $(-p,-q)$ mild solution of \eqref{eq5}.
Then there exists $q^\prime > q$ such that the mild solution is a $(-p,-q^\prime)$ strong solution.
\begin{proof}
From \eqref{eq7}, we have the equality in $\Sc_{-q}$
\[Y_s=S_s^\ast \psi+\int_0^s S^\ast _{s-r}A^\ast Y_r\cdot dB_r.\]
Note that, from Proposition \ref{prop1} we have the boundedness of the linear operator $L^\ast :\Sc_{-q}\cap\E^\prime(K)\rightarrow \Sc_{-q^\prime}\cap\E^\prime(K)$, for any $q^\prime > [q]+4$, where $K$ is some compact set in $\R^d$. In fact the same argument gives the boundedness of $L^\ast :\Sc_{-q}\rightarrow \Sc_{-q^\prime}$, for any $q^\prime > [q]+4$.\\
Hence, operating on both sides of the above equation by the linear
operator $L^\ast $, and integrating from $0$ to $t$ we obtain an
equality in $\Sc_{-q^\prime}$
\[\int_0^tL^\ast Y_sds=\int_0^tL^\ast S_s^\ast \psi ds+\int_0^t\int_0^sL^\ast S^\ast _{s-r}A^\ast Y_r\cdot dB_rds.\]
Now, by applying stochastic Fubini and integration by parts formulas respectively on the R.H.S. of the above equation we get
\begin{align*}
\int_0^t L^\ast Y_s ds&=\int_0^t L^\ast S_s^\ast \psi ds+\int_0^t\int_r^tL^\ast S^\ast _{s-r}A^\ast Y_rds \cdot dB_r\\
&=S^\ast _t\psi-\psi+\int_0^tS^\ast _{t-r}A^\ast Y_r\cdot dB_r-\int_0^tA^\ast Y_r\cdot dB_r\\
&=\left(S^\ast _t\psi+\int_0^tS^\ast _{t-r}A^\ast Y_r\cdot dB_r\right)-\psi-\int_0^tA^\ast Y_r\cdot dB_r\\
&=Y_t-\psi-\int_0^tA^\ast Y_r\cdot dB_r.
\end{align*}
Hence $Y_t=\psi+\int_0^tA^\ast Y_s\cdot dB_s+\int_0^tL^\ast Y_s ds$ in $\Sc_{-q^\prime}$. This completes the proof.
\end{proof}
\end{theorem}

\section{Uniqueness}\label{s:4}
We now consider the uniqueness of strong and mild solutions of \eqref{eq5}. The uniqueness condition, viz. the Monotonicity inequality, involves both the domain and the range of the operators $L^\ast, A^\ast_i$. For a compact subset $K$ of $\R^d$, it was shown in \cite{MR2373102} that $A^\ast :\Sc_{p}\cap\E^\prime(K)\rightarrow\mathcal{L}(\R^r,\Sc_{q}\cap\E^\prime(K))$ and $L^\ast :\Sc_{p}\cap\E^\prime(K)\rightarrow \Sc_{q}\cap\E^\prime(K)$ are bounded linear
operators, first when $p,q$ are both positive satisfing $p > [q]+4$ and then by duality when $p,q$ are negative satisfying $-q > [-p]+4$.

\begin{definition}[Monotonicity inequality, {\cite[equation
(4.2)]{MR2373102}}]\label{Mon-ineq} Fix $p,q$ both positive or both
negative, such that $A^\ast
:\Sc_{p}\cap\E^\prime\rightarrow\mathcal{L}(\R^r,\Sc_{q}\cap\E^\prime)$
and $L^\ast :\Sc_{p}\cap\E^\prime\rightarrow \Sc_{q}\cap\E^\prime$
are linear operators. Say that the pair of operators
$(L^\ast,A^\ast)$ satisfies the $(p,q)$ Monotonicity inequality if
\begin{equation}\label{mono-ineq}
2\langle\phi,L^\ast\phi\rangle_{q} + \sum\limits_{i=1}^r \| A_i^\ast\phi \|^2_{q}  \leq \; C_K \|\phi\|_{q}^2,
\; \forall \phi \in \Sc_{p}\cap\E^\prime(K),
\end{equation}
for all compact subsets $K$ of $\R^d$. Here $C_K$ is some positive constant depending on the set $K$.
\end{definition}

\begin{theorem}[{\cite[Theorem 4.4]{MR2373102}}]
Let $p \geq 0$ and $q > [p]+4$. If the $(-p,-q)$ Monotonicity inequality holds, then we have the uniqueness of $(-p,-q)$ strong solutions.
\end{theorem}
As a consequence of Theorem \ref{mild-strong}, we have the next result.

\begin{corollary}
Let $p, p^\prime, q$ be as in \eqref{L*}, \eqref{A*}, \eqref{S*}. Let $q^\prime$ be as in Theorem \ref{mild-strong}. If the $(-p,-q^\prime)$ Monotonicity inequality holds, then we have the uniqueness of $(-p,-q)$ mild solutions.
\end{corollary}

\begin{proof}
If $\{Y^1_t\}$ and $\{Y^2_t\}$ are two $(-p-q)$ mild solutions of \eqref{eq5}, then by Theorem \ref{mild-strong}, they are also $(-p,-q^\prime)$ strong solutions. If the $(-p,-q^\prime)$ Monotonicity inequality holds, then we have a.s. $Y^1_t = Y^2_t, \forall t \geq 0$ in $\Sc_{-q^\prime}$. Since both $\{Y^1_t\}$ and $\{Y^2_t\}$ are $\Sc_{-p}$ valued, we get the required uniqueness.
\end{proof}

We now describe a situation where the Monotonicity inequality holds. See \cite[Theorem 2.1]{MR2590157}, \cite[Theorem 4.6]{MR3331916} for other cases where this inequality holds.

\begin{theorem}\label{L2-Monotonicity}
Let $\sigma, b$ be as in \eqref{eq1}. Fix $p \geq 5$. Then $(L^\ast,A^\ast)$ satisfies the $(p,0)$ Monotonicity inequality. \end{theorem}

\begin{proof}
From the remark about the boundedness of $A^{\ast}$ and $L^{\ast}$
made before Definition \ref{Mon-ineq}, it is enough to verify the
inequality \eqref{mono-ineq} for $\phi \in C^\infty_c(K) \subset
{\Sc}_p\cap\E^{\prime}(K)$, for any compact set $K$. Consider the
operator $L_1^\ast\psi:=-\sum_{i=1}^d\partial_i(b_i\psi)$. Recall
from Section \ref{s:2} that we also use $\langle\cdot,\cdot\rangle$
for the inner product in ${\mathcal L}^2$. We have
\begin{align*}
\langle L_1^\ast\phi, \phi \rangle &= - \sum_{i=1}^d\langle \partial_ib_i\phi+b_i\partial_i\phi, \phi\rangle = \left\langle \left(-\sum_{i=1}^d
 \partial_ib_i\right)\phi,\phi \right\rangle - \sum_{i=1}^d\langle \partial_i \phi, b_i\phi\rangle\\
&=  \left\langle \left(-\sum_{i=1}^d \partial_ib_i\right)\phi,\phi
\right\rangle + \sum_{i=1}^d\langle \phi, \partial_i(b_i\phi)\rangle
\end{align*}
and hence
\begin{equation}\label{L1}
\langle L^\ast_1 \phi,\phi\rangle = -\frac{1}{2}\left\langle
\left(\sum_{i=1}^d \partial_ib_i\right)\phi,\phi \right\rangle \leq
C_b\|\phi\|^2,
\end{equation}
where $C_b$ is a positive constant depending on $b$. Recall that
$A^\ast \phi=(A^\ast _1\phi,\cdots,A^\ast _r\phi)$ with $A^\ast
_i\psi=-\sum_{j=1}^d\partial_j(\sigma_{ji}\psi)$. Also define
$L^\ast _2\phi
:=\frac{1}{2}\sum_{i,j=1}^d\partial^2_{i,j}((\sigma\sigma^t)_{ij}\phi)$.
For any $1 \leq i,j \leq d$, integration by parts yields
\begin{align*}
\left\langle \partial_j(\sigma_{jk})\phi, \sigma_{ik}\partial_i\phi \right\rangle &= \int_{\R^d} \partial_j(\sigma_{jk})\phi\, \sigma_{ik}\partial_i\phi\\
&= \frac{1}{2}\int_{\R^d} \partial_j(\sigma_{jk}) \sigma_{ik}\partial_i(\phi^2) = - \frac{1}{2}\int_{\R^d} \partial_i(\partial_j(\sigma_{jk}) \sigma_{ik})\phi^2
\end{align*}
Using above observation, we have
\begin{align*}
&\sum_{k=1}^r\|A_k^\ast\phi\|^2\\
&= \sum_{k=1}^r \left\langle \sum_{j=1}^d \partial_j(\sigma_{jk}\phi),\sum_{i=1}^d \partial_i(\sigma_{ik}\phi) \right\rangle\\
&= \sum_{k=1}^r \sum_{i,j=1}^d \left[ \left\langle \partial_j(\sigma_{jk})\phi, \partial_i(\sigma_{ik})\phi \right\rangle + \left\langle \partial_j(\sigma_{jk})\phi, \sigma_{ik}\partial_i\phi \right\rangle+ \left\langle \sigma_{jk}\partial_j\phi, \partial_i(\sigma_{ik})\phi \right\rangle\right]\\
&+ \sum_{k=1}^r \left\langle \sum_{j=1}^d \sigma_{jk} \partial_j\phi, \sum_{i=1}^d \sigma_{ik}\partial_i\phi \right\rangle\\
&= \sum_{k=1}^r \sum_{i,j=1}^d \left[ \left\langle \partial_j(\sigma_{jk})\phi, \partial_i(\sigma_{ik})\phi \right\rangle + 2\left\langle \partial_j(\sigma_{jk})\phi, \sigma_{ik}\partial_i\phi \right\rangle\right]\\
&+ \sum_{k=1}^r \left\langle \sum_{j=1}^d \sigma_{jk} \partial_j\phi, \sum_{i=1}^d \sigma_{ik}\partial_i\phi \right\rangle\\
&=\sum_{k=1}^r \sum_{i,j=1}^d  \int_{\R^d} \left[\partial_j(\sigma_{jk})\partial_i(\sigma_{ik})-\partial_i(\partial_j(\sigma_{jk}) \sigma_{ik})\right]\phi^2 + \sum_{k=1}^r \left\langle \sum_{j=1}^d \sigma_{jk} \partial_j\phi, \sum_{i=1}^d \sigma_{ik}\partial_i\phi \right\rangle\\
&=-\sum_{k=1}^r \sum_{i,j=1}^d  \int_{\R^d} \partial_{ij}^2(\sigma_{jk}) \sigma_{ik}\,\phi^2 + \sum_{k=1}^r \left\langle \sum_{j=1}^d \sigma_{jk} \partial_j\phi, \sum_{i=1}^d \sigma_{ik}\partial_i\phi \right\rangle
\end{align*}
Another integration by parts argument yields,
\begin{align*}
\langle L_2^\ast\phi, \phi\rangle&=\frac{1}{2}\sum_{i,j=1}^d\left\langle \partial_{ij}^2\left(\sum_{k=1}^r \sigma_{ik} \sigma_{jk} \phi\right), \phi \right\rangle\\
&= -\frac{1}{2}\sum_{i,j=1}^d \sum_{k=1}^r \left\langle \partial_j(\sigma_{ik} \sigma_{jk}) \phi + \sigma_{ik} \sigma_{jk} \partial_j\phi, \partial_i\phi \right\rangle\\
&=-\frac{1}{2}\sum_{i,j=1}^d \sum_{k=1}^r \int_{\R^d} \partial_j(\sigma_{ik} \sigma_{jk}) \phi \, \partial_i\phi-  \frac{1}{2}\sum_{i,j=1}^d \sum_{k=1}^r \left\langle \sigma_{ik} \sigma_{jk} \partial_j\phi, \partial_i\phi \right\rangle\\
&=-\frac{1}{4}\sum_{i,j=1}^d \sum_{k=1}^r \int_{\R^d} \partial_j(\sigma_{ik} \sigma_{jk})  \partial_i(\phi^2)-  \frac{1}{2}\sum_{i,j=1}^d \sum_{k=1}^r \left\langle \sigma_{ik} \sigma_{jk} \partial_j\phi, \partial_i\phi \right\rangle\\
&=-\frac{1}{4}\sum_{i,j=1}^d \sum_{k=1}^r \int_{\R^d} \partial_{ij}^2(\sigma_{ik} \sigma_{jk})\phi^2-  \frac{1}{2}\sum_{i,j=1}^d \sum_{k=1}^r \left\langle \sigma_{ik} \sigma_{jk} \partial_j\phi, \partial_i\phi \right\rangle
\end{align*}
Then
\begin{equation}\label{L2-A}
2\langle L_2^\ast\phi, \phi\rangle + \sum_{k=1}^r\|A_k^\ast\phi\|^2 \leq C_{\sigma,K}\|\phi\|^2,
\end{equation}
where $C_{\sigma,K}$ is a positive constant depending on $\sigma$ and $K$. Now adding \eqref{L1} and \eqref{L2-A} together, we get the required inequality.
\end{proof}

As an application of Theorem \ref{L2-Monotonicity} we get the next
result, wherein we note that the initial value $\psi$ need no longer
be of compact support.

\begin{corollary}
Let $\psi \in \Sc$ and $p \geq 5$. We have the existence and uniqueness of $(p,0)$ strong solutions of \eqref{eq5}. \end{corollary}
\begin{proof}
Let $\psi$ be a $C^\infty_c(\R^d)$ function. Recall that $\{Z_t(\psi)\}$ is $\Sc_p\cap\E^\prime$ valued for $p \geq 5$, is locally of compact support and solves \eqref{eq5} ((see Example \ref{spl-example}) and Theorem \ref{thm3}). From Theorem \ref{L2-Monotonicity}, we get the uniqueness.

Since $C^\infty_c(\R^d)$ is dense in
$\Sc$ (in $\|\cdot\|_p$), by density arguments, the result follows.
\end{proof}
{\bf Acknowledgement :} The first author would like to acknowledge the fact that he was supported by the NBHM (National Board of Higher Mathematics, under Department of Atomic Energy, Government of India) Post Doctoral Fellowship. The second author would like to thank
P.Fitzsimmons for some discussions relating to mild solutions of
SPDE's. He would also like to thank V.Mandrekar and L.Gawarecki for
discussions relating to the notion of distribution valued processes
that are `locally of compact support'. The third author would like to acknowledge the fact that he was supported by the ISF-UGC research grant.

\end{document}